\documentclass[a4paper]{amsart}
\usepackage{amssymb}
\usepackage{amscd}
\usepackage{amsthm}
\usepackage{amsmath}
\usepackage{latexsym}
\usepackage[all]{xy}
\usepackage[utf8]{inputenc}
\usepackage{enumitem}
\DeclareMathOperator{\PSupp}{PSupp}

\setlist[enumerate]{label={(\roman*)}}

\theoremstyle{plain}
\newtheorem{theorem}{Theorem}
\newtheorem{corollary}[theorem]{Corollary}
\newtheorem{lemma}[theorem]{Lemma}
\newtheorem{proposition}[theorem]{Proposition}

\theoremstyle{definition}
\newtheorem{definition}[theorem]{Definition}
\newtheorem{example}[theorem]{Example}

\theoremstyle{remark}

\newtheorem*{remark}{Remark}

\numberwithin{theorem}{section} 

\newcommand{\Spec}[1]{\operatorname{Spec}(#1)}
\newcommand{\mSpec}[1]{\operatorname{mSpec}(#1)}
\newcommand{\rad}[1]{\operatorname{rad}(#1)}
\newcommand{\card}{\mbox{\rm{card\,}}}

\newcommand{\Ass}[2]{\mbox{\rm{Ass}}_ {#1}(#2)}

\newcommand{\Hom}[3]{\operatorname{Hom}_{#1}(#2,#3)}
\newcommand{\Ext}[4]{\operatorname{Ext}^{#1}_{#2}(#3,#4)}

\newcommand{\rmod}[1]{\mbox{\rm{Mod}--}{#1}}

\newcommand{\VF}{\mathcal{VF}}
\newcommand{\CA}{\mathcal{CA}}
\newcommand{\LV}{\mathcal{LV}}

\begin{document}

\title{Very flat, locally very flat, and contraadjusted modules}
\author{\textsc{Alexander Sl\' avik and Jan Trlifaj}}
\address{Charles University, Faculty of Mathematics and Physics, Department of Algebra \\
Sokolovsk\'{a} 83, 186 75 Prague 8, Czech Republic}
\email{Slavik.Alexander@seznam.cz}
\email{trlifaj@karlin.mff.cuni.cz}

\date{\today}
\thanks{Research supported by GA\v CR 14-15479S and GAUK 571413}
\begin{abstract} Very flat and contradjusted modules naturally arise in algebraic geometry in the study of contraherent cosheaves over schemes. Here, we investigate the structure and approximation properties of these modules over commutative noetherian rings. Using an analogy between projective and flat Mittag-Leffler modules on one hand, and very flat and locally very flat modules on the other, we prove that each of the following statements are equivalent to the finiteness of the Zariski spectrum $\Spec R$ of a noetherian domain $R$: (i)~the class of all very flat modules is covering, (ii)~the class of all locally very flat modules is precovering, and (iii)~the class of all contraadjusted modules is enveloping. We also prove an analog of Pontryagin's Criterion for locally very flat modules over Dedekind domains.
\end{abstract}

\maketitle

\section*{Introduction} 

Very flat and contraadjusted modules have recently been introduced by Positselski \cite{P} in order to study instances of the comodule-contramodule correspondence for quasi-coherent sheaves and contraherent cosheaves over schemes. 

Recall \cite{EE} that given a scheme $X$ with the structure sheaf $\mathcal{O}_X$, a quasi-coherent sheaf $Q$ on $X$ can be viewed as a representation assigning 
\begin{itemize}
\item to every affine open subscheme $U \subseteq X$, an $\mathcal{O}_X(U)$-module $Q(U)$ of sections, and 
\item to each pair of embedded affine open subschemes $V \subseteq U \subseteq X$, 
an $\mathcal{O}_X(U)$-homomorphism $f_{UV}: Q(U) \to Q(V)$ such that

$$\mbox{id}_{\mathcal{O}_X(V)} \otimes f_{UV} : \mathcal{O}_X(V) \otimes_{\mathcal{O}_X(U)} Q(U) \to \mathcal{O}_X(V) \otimes_{\mathcal{O}_X(U)} Q(V) \cong Q(V)$$

is an $\mathcal{O}_X(V)$-isomorphism, and $f_{UV} f_{VW} = f_{UW}$ for $W \subseteq V \subseteq U \subseteq X$. 
\end{itemize} 

This kind of representation makes it possible to transfer various module theoretic notions to quasi-coherent sheaves on $X$. For example, (infinite-dimensional) vector bundles correspond thus to those representations where each $\mathcal{O}_X(U)$-module $Q(U)$ is (infinitely generated) projective. Notice that the functors $\mathcal{O}_X(V) \otimes_{\mathcal{O}_X(U)} -$ are exact, that is, all the $\mathcal{O}_X(U)$-modules $\mathcal{O}_X(V)$ are flat. 

Not all affine open subschemes are needed for the representation above: a set of them, $\mathcal S$, covering both $X$, and all $U \cap V$ where $U, V \in \mathcal S$, will do. The set $\mathcal S$ can often be small, making the representation above more efficient. However, when transferring module theoretic notions to quasi-coherent sheaves in this way, one needs to prove independence from the representation (i.e., from the choice of the open affine covering $\mathcal S$ of $X$). This is a non-trivial task even for the notion of a vector bundle, cf.\ \cite{RG}. 

Modern approach to cohomology theory of quasi-coherent sheaves on a scheme $X$ is based on the study of their unbounded derived category. By the classic work of Quillen, this reduces to studying model category structures on the category of unbounded chain complexes of quasi-coherent sheaves. Hovey's work \cite{H} has shown that the latter task reduces further to studying complete cotorsion pairs in the category of (complexes of) quasi-coherent sheaves. So eventually, one is faced with problems concerning approximations (precovers and preenvelopes) of modules. 

While it is obvious that projective modules form a precovering class, and flat modules are known to form a covering class for more than a decade \cite{BEE}, the surprising fact that flat Mittag-Leffler modules over non-perfect rings do not form a precovering class is quite recent, see \cite{AST}.    
  
\medskip
In \cite{P}, a dual representation was used to define \emph{contraherent cosheaves} $P$ on $X$ as the representations assigning 
\begin{itemize}
\item to every affine open subscheme $U \subseteq X$, of an $\mathcal{O}_X(U)$-module $P(U)$ of cosections, and 
\item to each pair of embedded affine open subschemes $V \subseteq U \subseteq X$,  
an $\mathcal{O}_X(U)$-homomorphism $g_{VU}: P(V) \to P(U)$ such that
$$\mbox{Hom}_{\mathcal{O}_X(U)}(\mathcal{O}_X(V),g_{VU}) : P(V) \to \mbox{Hom}_{\mathcal{O}_X(U)}(\mathcal{O}_X(V),P(U))$$
is an $\mathcal{O}_X(V)$-isomorphism, and $g_{WV} g_{VU} = g_{WU}$ for $W \subseteq V \subseteq U \subseteq X$. 
\end{itemize} 

Since the $\mathcal{O} _X(U)$-module $\mathcal{O} _X(V)$ is flat, but not projective in general, the Hom-functor above need not be exact. Its exactness is forced by imposing the following additional condition on the contraherent cosheaf $P$: 
\begin{itemize}
\item $\mbox{Ext}^1_{\mathcal{O}_X(U)}(\mathcal{O}_X(V),P(U)) = 0.$
\end{itemize}

In \cite{P}, a hitherto unnoticed additional property of the $\mathcal{O} _X(U)$-modules $\mathcal{O} _X(V)$ has been discovered: these modules are \emph{very flat} in the sense of Definition \ref{veryf} below. Indeed, by \cite[1.2.4]{P}, if $R \to S$ is a homomorphism of commutative rings such that the induced morphism of affine schemes $\mbox{Spec}(S) \to \mbox{Spec}(R)$ is an open embedding, then $S$ is a very flat $R$-module.
It follows that for each contraherent cosheaf $P$, the $\mathcal{O}_X(U)$-module $P(U)$ is \emph{contraadjusted} (again, see  Definition \ref{veryf} below). Moreover, the notion of a very flat module is local for affine schemes \cite[1.2.6]{P}.  

\medskip
One can use the representations above and extend various module theoretic notions to contraherent cosheaves on $X$. However, one first needs to understand the algebraic part of the picture. This is our goal here: we study in more detail the structure of very flat, locally very flat, and contraadjusted modules over commutative rings, as well as their approximation properties. 

We pursue the analogy between projective and flat Mittag-Leffler modules on one hand, and very flat and locally very flat modules on the other, in order to trace non-existence of precovers to the latter setting. Our main results are proved in the case when $R$ is a noetherian domain: in Theorems \ref{VFcoverdomain} and \ref{char}, we show that the class of all very flat modules is covering, iff  the class of all locally very flat modules is precovering, iff the Zariski spectrum of $R$ is finite. Moreover, in Corollary \ref{caenveloping}, we show that this is further equivalent to the class of all contraadjusted modules being enveloping. In the particular setting of Dedekind domains, we provide in Theorem \ref{variants} a characterization of locally very flat modules analogous to Pontryagin's Criterion for $\aleph_1$-freeness 
(cf.\ \cite[Theorem IV.2.3]{EM}).

\section{Preliminaries}

In this paper, $R$ denotes a commutative ring, and $\rmod R$ the category of all ($R$-) modules. Let $M$ be a module. We will use the notation $M \trianglelefteq N$ to indicate that $M$ is an essential submodule in a module $N$, and $E(M)$ will denote the injective envelope of $M$. 

A major theme of the classic module theory consists in finding direct sum decompositions of modules, preferably into direct sums of small, or well-understood types of modules. More in general, one can aim at deconstructions of modules, that is, at expressing them as transfinite extensions rather than direct sums:     

\begin{definition}\label{filt} Let $\mathcal C$ be a class of modules. A module $M$ is said to be \emph{$\mathcal C$-filtered} (or a \emph{transfinite extension} of the modules in $\mathcal C$), provided that there exists an increasing chain $\mathcal M  = ( M_\alpha \mid \alpha \leq \sigma )$ of submodules of $M$ with the following properties: $M_0 = 0$, $M_\alpha = \bigcup_{\beta < \alpha} M_\beta$ for each limit ordinal $\alpha \leq \sigma$, $M_{\alpha +1}/M_\alpha \cong C_\alpha$ for some $C_\alpha \in \mathcal C$ for each $\alpha < \sigma$, and $M_\sigma = M$.  

The chain $\mathcal M$ is called a \emph{$\mathcal C$-filtration} of the module $M$ of length $\sigma$.    
\end{definition}

If a module possesses a $\mathcal C$-filtration, then there are other $\mathcal C$-filtrations at hand, and one can replace the original filtration by the one more appropriate to a particular problem. The abundance of $\mathcal C$-filtrations follows from the next result going back to Hill:

\begin{lemma}{\rm (\cite[Theorem 7.10]{GT})}\label{hill} 
Let $R$ be a ring, $M$ a module, $\kappa$ a regular infinite cardinal, and $\mathcal C$ a class of $< \kappa$--presented modules. Let $\mathcal M = (M_\alpha \mid \alpha \leq \sigma )$ be a $\mathcal C$-filtration of $M$.

Then there exists a family $\mathcal H$ consisting of submodules of $M$ such that 
\begin{itemize}
\item[\rm{(i)}] $\mathcal M \subseteq \mathcal H$,
\item[\rm{(ii)}] $\mathcal H$ forms a complete distributive sublattice of the complete modular lattice of all submodules of $M$, 
\item[\rm{(iii)}] $P/N$ is $\mathcal C$-filtered for all $N \subseteq P$ in $\mathcal H$, and 
\item[\rm{(iv)}] If $N \in \mathcal H$ and $S$ is a subset of $M$ of cardinality $< \kappa$, then there is $P \in \mathcal H$ such that $N \cup S \subseteq P$ and $P/N$ is $< \kappa$--presented.
\end{itemize}
\end{lemma}

$\mathcal C$-filtrations are closely related to approximations (precovers and preenvelopes) of modules:

\begin{definition} 
\begin{itemize}
\item[\rm{(i)}] A class of modules $\mathcal A$ is \emph{precovering} 
if for each module $M$ there is $f \in \mbox{Hom}_R(A,M)$ with $A \in \mathcal A$ such that
each $f^\prime \in \mbox{Hom}_R(A^{\prime},M)$ with $A^\prime \in \mathcal A$ has a factorization through 
$f$: 
\[
\xymatrix{A \ar[r]^{f} & M \\ 
{A^\prime} \ar@{-->}[u]^{g} \ar[ur]_{f^\prime} &}
\]
The map $f$ is called an \emph{$\mathcal A$-precover} of $M$ (or a \emph{right $\mathcal A$-approximation} of $M$).
\item[\rm{(ii)}] An $\mathcal A$-precover is \emph{special} in case it is surjective, and its kernel $K$ satisfies 
$\mbox{Ext}_R^1(A,K) = 0$ for each $A \in \mathcal A$. 
\item[\rm{(iii)}] Let $\mathcal A$ be precovering. Assume that in the setting of (i), if $f^\prime = f$ then each factorization $g$ is an automorphism. Then $f$ is an \emph{$\mathcal A$-cover} of $M$. $\mathcal A$ is called a \emph{covering} class in case each module has an $\mathcal A$-cover. We note that each covering class containing the projective modules and closed under extensions is necessarily special precovering (Wakamatsu Lemma). 

For example, the class of all projective modules is easily seen to be precovering, while the class of all flat modules is covering (by the Flat Cover Conjecture proved in \cite{BEE}). By a classic result of Bass, the class of all projective modules is covering, iff it coincides with the class of all flat modules, i.e., iff $R$ is a right perfect ring. 

Dually, we define \emph{(special) preenveloping} and \emph{enveloping} classes of modules. For example, the class of all injective modules is an enveloping class. 
\end{itemize}   
\end{definition} 

Cotorsion pairs are a major source of approximations. Moreover, by a classic result of Salce, they provide for an explicit duality between special precovering and special preenveloping classes of modules:

\begin{definition}\label{Salce}
A pair of classes of modules $\mathfrak C = (\mathcal A, \mathcal B)$ is a \emph{cotorsion pair} provided that  
\begin{enumerate}
\item $\mathcal A = {}^\perp \mathcal B := \{ A \in \mbox{Mod-}R \mid \mbox{Ext}^1_R(A,B) = 0 \mbox{ for all } B \in \mathcal B \}$, and
\item $\mathcal B = \mathcal A ^\perp := \{ B \in \mbox{Mod-}R \mid \mbox{Ext}^1_R(A,B) = 0 \mbox{ for all } A \in \mathcal A \}$. 
\end{enumerate}
If moreover \emph{$3.$ For each module $M$, there exists an exact sequences $0 \to B \to A \to M \to 0$ with $A \in \mathcal A$ and $B \in \mathcal B$}, then $\mathfrak C$ is called \emph{complete}.
\end{definition}

Condition $3.$ implies that $\mathcal A$ is a special precovering class. In fact, $3.$ is equivalent to its dual:
\emph{$3^\prime.$ For each module $M$ there is an exact sequence $0 \to M \to B \to A \to 0$ with $A \in \mathcal A$ and $B \in \mathcal B$}, which in turn implies that $\mathcal B$ is a special preenveloping class.

Module approximations are abundant because of the following basic facts (for their proofs, see e.g.\ \cite{GT}):

\begin{theorem}\label{approx-main} Let $\mathcal S$ be a set of modules.
\begin{enumerate}
\item Let $\mathcal C$ denote the class of all $\mathcal S$-filtered modules. 
Then $\mathcal C$ is precovering. 
Moreover, if $\mathcal C$ is closed under direct limits, then $\mathcal C$ is covering.
\item The cotorsion pair $(^\perp (\mathcal S ^\perp), \mathcal S ^\perp)$ is complete (this is the cotorsion pair \emph{generated} by the set $\mathcal S$).

Moreover, if $R \in \mathcal S$, then the special precovering class $\mathcal A := {}^\perp (\mathcal S ^\perp)$ coincides with the class of all direct summands of $\mathcal S$-filtered modules. If $\kappa$ is a regular uncountable cardinal such that each module in $\mathcal S$ is $< \kappa$-presented, and $\mathcal C$ denotes the class of all $< \kappa$-presented modules from $\mathcal A$, then $\mathcal A$ also coincides with the class of all $\mathcal C$-filtered modules. 
\end{enumerate}
\end{theorem}

For example, if $\mathcal S = \{ R \}$, then $\mathcal A$ is the class of all projective modules, and (2) gives that each projective module is a direct summand of a free one, and (for $\kappa = \aleph_1$) that each projective module is a direct sum of countably generated modules (Kaplansky Theorem). 

Relations between projective and flat Mittag-Leffler modules are the source of another generalization:

\begin{definition}\label{locally} A system $\mathcal S$ consisting of countably presented submodules of a module $M$ is a \emph{dense system} provided that $\mathcal S$ is closed under unions of well-ordered countable ascending chains, and each countable subset of $M$ is contained in some $N \in \mathcal S$.

Let $\mathcal C$ be a set of countably presented modules. Denote by $\mathcal A$ the class of all modules possessing a countable $\mathcal C$-filtration. A module $M$ is \emph{locally $\mathcal C$-free} provided that $M$ contains a dense system of submodules consisting of modules from $\mathcal A$. 
(Notice that if $M$ is countably presented, then $M$ is locally $\mathcal C$-free, iff $M \in \mathcal A$.) 
\end{definition}

For example, if $\mathcal C$ is a representative set of the class of all countably generated projective modules, then locally $\mathcal C$-free modules coincide with the flat Mittag-Leffler modules. The surprising fact that this class is not precovering in case $R$ is not a perfect ring has recently been proved by \v Saroch in \cite{AST}. The key obstruction for existence of flat Mittag-Leffler approximations are the Bass modules: 

\begin{definition}\label{Bassm}  Let $\mathcal C$ be a set of countably presented modules. A module $B$ is a \emph{Bass module} over $\mathcal C$ provided that $B$ is a countable direct limit of some modules from $\mathcal C$.

W.l.o.g., such $B$ is the direct limit of a chain 
$$C_0 \overset{f_0}\to C_1 \overset{f_1}\to \dots \overset{f_{i-1}}\to C_i \overset{f_i}\to C_{i+1} \overset{f_{i+1}}\to \dots$$
with $C_i \in \mathcal C$ and $f_i \in \Hom R{C_i}{C_{i+1}}$ for all $i < \omega$.
\end{definition}

\begin{example}\label{nperf} If $\mathcal C$ denotes the representative set of all finitely generated projective modules, then the Bass modules over $\mathcal C$ coincide with the countably presented flat modules. 
If $R$ is not right perfect, then a classic instance of such a Bass module $B$ arises when $C_i = R$ and $f_i$ is the left multiplication by $a_i$ ($i < \omega$), where 
$Ra_0 \supsetneq \dots \supsetneq Ra_n\dots a_0 \supsetneq Ra_{n+1}a_n\dots a_0 \supsetneq \dots$ is strictly decreasing chain of principal left ideals in $R$.
\end{example}

\begin{lemma}{\rm (\cite[Lemma 3.2]{AST})}\label{saroch}  Let $\mathcal C$ be a class of countably presented modules, 
and $\mathcal A$ the class of all locally $\mathcal C$-free modules. Assume there exists 
a Bass module $B$ over $\mathcal C$ such that $B$ is not a direct summand in a module from $\mathcal A$.  
Then $B$ has no $\mathcal A$-precover.
\end{lemma}

Note that in the setting of Example \ref{nperf}, Lemma \ref{saroch} yields that for each non-right perfect ring, the classic Bass module $B$ does not have a flat Mittag-Leffler precover. For further applications combining Lemma \ref{saroch} with (infinite dimensonal) tilting theory, we refer to \cite{AST}; our applications here will go in a different direction (see Lemma \ref{nprec} below).   

\medskip
We will also need the notion of the rank of a torsion-free module: recall that a module $M$ is \emph{torsion-free} provided that no non-zero element of $M$ is annihilated by any regular element (= non-zero-divisor) of $R$. 

First, we consider a classic particular case, when $R$ is a domain. We will denote by $Q$ the quotient field of $R$. For a torsion-free module $M$, $r(M)$ will denote its \emph{rank} defined by $r(M) = \dim_Q(M \otimes_R Q)$. Notice that $\kappa = r(M)$, iff $M$ is isomorphic to a module $M^\prime$ such that $R^{(\kappa)} \trianglelefteq M^\prime \trianglelefteq Q^{(\kappa)}$. 

Also, for each $0 \neq r \in R$, the localization $R[r^{-1}]$ coincides with the subring of $Q$ containing $R$ and consisting of (equivalence classes of) the fractions whose denominators are powers of $r$. In particular, $R[r^{-1}] \otimes_R R[s^{-1}] \cong R[s^{-1}]$ in case $r$ divides $s$. 

In Section \ref{sect1}, we will work in the more general setting of (commutative) rings whose prime radical $N = \rad R$ is nilpotent, and 
$\bar R = R/N$ is a Goldie ring, i.e., $\bar R$ has a semisimple classical quotient ring $\bar Q$, cf.\ \cite[Theorem 6.15]{GW}. In this setting, we will employ the notion of a reduced rank from \cite[p.194, Exercise 11G]{GW}: Let $n$ denote the nilpotency index of $N$. For a module $M$, we consider the chain $0 = MN^n \subseteq MN^{n-1} \subseteq \dots \subseteq MN \subseteq MN^0 = M$. Then $MN^{i}/MN^{i+1}$ is a $\bar R$-module for each $i < n$. The (reduced) \emph{rank} of $M$ is defined by $r(M) = \sum_{i < n} \ell(MN^{i}/MN^{i+1})$, where for a $\bar R$-module $P$, $\ell(P)$ denotes the composition length of the $\bar Q$-module $P \otimes_R \bar Q$. 

Note that this more general setting also includes the important particular case when $R$ is noetherian. Moreover, for torsion-free modules over domains, the notions of a reduced rank and rank coincide, so our notation is consistent. 

By \cite[Exercise 11G(b)]{GW}, if $0 \to M \to M^\prime \to M^\prime/M \to 0$ is exact and $M^\prime$ has finite rank, then 
$r(M^\prime) = r(M) + r(M^\prime/M)$ (i.e., the reduced rank is additive on short exact sequences). Moreover, $r(M) = 0$, if and only if $M$ is $S$-torsion where $S$ is the set of all $s \in R$ such that $s + \rad R$ is regular in $R/\rad R$.     

By \cite[Exercises 11.H and 11.I]{GW}, the following more general version of Small's Theorem \cite[Theorem 11.9]{GW} holds true:

\begin{theorem}\label{small} Let $R$ be a ring. Then $R$ has a classical quotient ring which is artinian, if and only if $N = \rad R$ is nilpotent, $\bar R = R/N$ is a Goldie ring, $r(R)$ is finite, and for each $r \in R$, $r$ is regular in $R$ iff $r + N$ is regular in $\bar R$.
\end{theorem}   

Finally, we note that in the more general setting $R[r^{-1}] \otimes_R R[s^{-1}] \cong R[(rs)^{-1}]$ for all $r, s \in R$, and  $R[r^{-1}] = 0$ iff $r \in \rad R$ (i.e., $r$ is nilpotent). 

\section{Very flat modules}\label{sect1} 

For each ring $R$, the class $\mathcal F$ of all flat modules fits in the complete cotorsion pair $(\mathcal F, \mathcal C)$, where $\mathcal C = \mathcal F ^\perp$ is the class of all \emph{cotorsion} modules. Very flat modules are also defined using complete cotorsion pairs:       

\begin{definition}\label{veryf} \rm A module $M$ is \emph{very flat}, provided that $M \in \mathcal{VF}$ where $(\mathcal{VF}, \mathcal{CA})$ denotes the complete cotorsion pair generated by the set  
$$\mathcal L = \{ R[r^{-1}] \mid  r \in R \}$$ 
and $R[r^{-1}]$ is the localization of $R$ at the multiplicative set $\{ 1, r, r^2, \dots \}$. The modules in the class $\mathcal{CA}$ are called \emph{contraadjusted}, \cite[\S1.1]{P}.
\end{definition}

Clearly, each projective module is very flat. Since the localization $R[r^{-1}]$ is a flat module for each $r \in R$, all very flat modules are flat, and hence each cotorsion module is contraadjusted. We postpone our investigation of contraadjusted modules to Section \ref{sect-CA}, and start with a more a precise description of very flat modules:
 
\begin{lemma}\label{morep}  
Each very flat module has projective dimension $\leq 1$. Moreover, $\mathcal{VF}$ coincides with the class of all direct summands of $\mathcal L$-filtered modules, and also with the class of all $\mathcal C$-filtered modules, where $\mathcal C$ is the class of all countably presented very flat modules. Each countably generated very flat module $M$ is a direct summand in a module possessing an $\mathcal L$-filtration of length $\sigma$, where $\sigma$ is a countable ordinal; in particular, $M$ is countably presented.   
\end{lemma}
\begin{proof} For the first claim, note that for each $r \in R$, the module $R[r^{-1}]$ is the direct limit of the direct system $R \overset{f_r}\to R \overset{f_r}\to R \to \dots $ where $f_r(1) = r$, so there is an exact sequence 
\begin{equation}\label{present}
0 \to R^{(\omega)} \overset{g_r}\to R^{(\omega)} \to R[r^{-1}] \to 0
\end{equation}
where $g_r(1_i) = 1_i -  1_{i+1}\cdot r$ for each $i < \omega$ and $( 1_i \mid i < \omega )$ denotes the canonical free basis of $R^{(\omega)}$. This shows that $R[r^{-1}]$ is countably presented, and has projective dimension $\leq 1$. The latter property extends to each (direct summand of an) $\mathcal L$-filtered module. 

The second claim follows from Theorem \ref{approx-main}(2).

Finally, if $M$ is a countably generated very flat module, then $M$ is a direct summand in a module $N$ possessing an $\mathcal L$-filtration $(N_\alpha \mid \alpha \leq \sigma )$. By the Hill Lemma \ref{hill}, we can modify the filtration so that $M \subseteq N_\tau$ for a countable ordinal $\tau \leq \sigma$.
\end{proof}
 
We continue with some more specific observations in the particular cases of domains, and of the noetherian rings possessing artinian classical quotient rings: 

\begin{lemma}\label{general}  
\begin{itemize}
\item[(i)] Assume that $R$ has an artinian classical quotient ring. 
Let $M$ be a submodule of a very flat module such that $r(M) = t < \infty$. Then there exist a finite sequence of non-nilpotent elements $\{ s_i \mid i < n \}$ of $R$ and a strictly increasing chain  
$$0 = M_0 \subsetneq M_1 \subsetneq \dots \subsetneq M_{n-1} \subsetneq M_n = M$$
such that $M_{i+1}/M_i$ is isomorphic to a submodule of $R[s_i^{-1}]$ for each $i < n$. 
Moreover, $t = n$ in case $R$ is a domain.
\item[(ii)] Assume that $R$ is a noetherian ring which has an artinian classical quotient ring. Let $M$ be a non-zero very flat module with $r(M) = t < \infty$. Then there exists $s \in R$ such that $M \otimes_R R[s^{-1}]$ is a non-zero finitely generated projective $R[s^{-1}]$-module. If $R$ is moreover a domain, then $M \otimes_R R[s^{-1}]$ has rank $t$.
\item[(iii)] Assume $R$ is a domain. Then $Q$ is very flat, iff $Q = R[s^{-1}]$ for some $0 \neq s \in R$. In this case, $Q$ has projective dimension $1$. 
\end{itemize}
\end{lemma}
\begin{proof} (i) By assumption, $M$ is a submodule in an $\mathcal L$-filtered module $P$. Let $\mathcal P = ( P_\alpha \mid \alpha \leq \sigma )$ be an $\mathcal L$-filtration of $P$. For each $\alpha \leq \sigma$, let $M_\alpha = M \cap P_\alpha$. Then the consecutive factor $M_{\alpha +1}/M_\alpha$ is isomorphic to a submodule of $P_{\alpha +1}/P_\alpha \in \mathcal L$. 

If $R$ is a domain, then exactly $t$ of these consecutive torsion-free factors are non-zero (and of rank $1$). So the chain has exactly $t+1$ distinct terms, $0 = M_0 \subsetneq M_1 \subsetneq \dots \subsetneq M_{t-1} \subsetneq M_t = M$. 

In the general case, by Theorem \ref{small}, the elements of $R$ regular modulo the prime radical coincide with the regular elements of $R$. If $s \in R$ is not nilpotent and $0 \neq M^\prime \subseteq R[s^{-1}]$, then $M^\prime$ is a torsion-free module, whence $r(M^\prime) > 0$. So if the consecutive factor $M_{\alpha +1}/M_\alpha \subseteq R[s_\alpha^{-1}]$ is non-zero, then it has non-zero reduced rank. The additivity of the reduced rank yields that there are only finitely many such non-zero consecutive factors, and the claim follows.       

(ii) For each $i < n$, let $s_i$ be a (non-nilpotent) element of $R$ such that $0 \neq M_{i+1}/M_i$ embeds into $R[s_i^{-1}]$. Let $X_i = \{ i_0, i_1, \dots , i_k \}$ be a subset of $n$ such that $i \in X_i$, $0 \neq (M_{i+1}/M_i) \otimes_R R[s_{i_0}^{-1}] \otimes_R \dots \otimes_R R[s_{i_k}^{-1}] \cong (M_{i+1}/M_i) \otimes_R R[p_i^{-1}]$ for $p_i = \prod_{j \in X_i} s_j$, but $(M_{i+1}/M_i) \otimes_R R[p_i^{-1}] \otimes_R R[s_j^{-1}] = 0$ for each $j \notin X_i$. 

Consider $k < n$ such that $X_k$ is maximal, that is, $X_k$ is not properly contained in $X_i$ for any choice of $X_i$ as above and any $i < n$. Then for each $i \in X_k$, $I_i = (M_{i+1}/M_i) \otimes_R R[p_k^{-1}] \subseteq R[s_i^{-1}] \otimes_R R[p_k^{-1}] \cong R[p_k^{-1}]$ is isomorphic to a (finitely generated) ideal of the noetherian ring $R[p_k^{-1}]$. Moreover, $I_k \neq 0$ by the definition of $X_k$. 

If $j \notin X_k$, then $I_j = (M_{j+1}/M_j) \otimes_R R[p_k^{-1}] = 0$: Otherwise, there is $0 \neq x = r/s_j^m \in M_{j+1}/M_j \subseteq R[s_j^{-1}]$ such that $x$ is not annihilated by any power of $p_k$, whence $x$ is not annihilated by any power of $p_k.s_j$, too. So $I_j \otimes_R R[s_j^{-1}] \neq 0$, which implies that we can choose $X_j$ so that $X_k \subseteq X_j$ and $j \in X_j \setminus X_k$, in contradiction with the maximality of $X_k$. 

It follows that the $R[p_k^{-1}]$-module $M \otimes_R R[p_k^{-1}]$ is $\{ I_i \mid i \in X_k \}$-filtered. Since $M$ is very flat, putting $s = p_k$, we conclude that $M \otimes_R R[s^{-1}]$ is non-zero, finitely generated, and flat (even very flat, \cite[1.2.2]{P}), hence a non-zero projective $R[s^{-1}]$-module. 

Moreover, if $R$ is a domain, then $X_k = n = t$ and $I_i \neq 0$ for each $i \in X_k$, whence $M \otimes_R R[s^{-1}]$ has rank $t$. 

(iii) If $Q \in \mathcal{VF}$, then since $Q$ has rank $1$, the chain $\mathcal N$ constructed in part (i) consists only of two elements, $0$ and $Q$, and $Q$ is a submodule in $R[s^{-1}]$ for some $0 \neq s \in R$, whence 
$Q = R[s^{-1}]$. The latter equality clearly implies $Q \in \mathcal{VF}$, whence $Q$ has projective dimension $1$ by Lemma \ref{morep}.  
\end{proof}

\begin{example}\label{non-art} In the case when $R$ is noetherian, but $R$ does not have artinian classical quotient ring, there may exist non-zero torsion-free submodules of $R[s^{-1}]$ whose rank is zero. So the argument in the proof of Lemma \ref{general}(i) involving additivity of the rank does not apply.

For an example, consider the ring $R = k[x,y]/I$ where $I = (xy,y^2)$ (see \cite[p.193]{GW}). Then the prime radical $N = \rad R$ of $R$ is generated by $y + I$. Moreover, $x$ is regular modulo $N$ (but it is not regular in $R$). Since $x$ annihilates $N$, \cite[Lemma 11.5]{GW} implies $r(N) = 0$, though $N$ is a non-zero (in fact, simple) torsion-free submodule of $R$.  
\end{example}    

In Proposition \ref{frank}(ii) below, we will see that if $M$ is of finite rank over a Dedekind domain $R$, then the converse of Lemma \ref{general}(ii) holds, that is, $M$ is very flat, iff there exists $0 \neq s \in R$ such that $M \otimes_R R[s^{-1}]$ is a projective $R[s^{-1}]$-module. However, the converse fails already for rank one modules over regular domains of Krull dimension $2$:

\begin{example}\label{nconverse}
Let $k$ be a field, $R = k[x, y]$, and
\[ \textstyle M = R\bigl[\frac xy, \frac{x^2}{y^2}, \ldots \bigr]. \]
Clearly, $R \subseteq M \subseteq Q$, so $M$ has rank $1$, and $M \otimes_R R[y^{-1}] \cong R[y^{-1}]$ is a free $R[y^{-1}]$-module of rank $1$.
We will show that $M$ is not even a flat module.

Let $m$ be the maximal ideal generated by the elements $x$ and $y$. Our goal is to show that the inclusion $m \to R$ is not injective after tensoring by $M$; namely, the element $x \otimes 1 - y \otimes (x/y) \in m \otimes_R M$ is nonzero, but maps to zero in the module $R \otimes_R M \cong M$. The latter being clear, let us verify the former: If $x \otimes 1 - y \otimes (x/y) = 0$ in $m \otimes_R M$, then (using the criterion for vanishing of an element of a tensor product, cf.\ \cite[Proposition I.8.8]{S}) there are $r_0, r_1, \ldots, s_0, s_1, \ldots \in R$, all but finitely many equal to zero, such that
\begin{align}
1 &= \textstyle\sum_{i<\omega} r_i \frac{x^i}{y^i},\label{nconverse1}\\
-\tfrac xy &= \textstyle\sum_{i<\omega} s_i \frac{x^i}{y^i}\label{nconverse2}
\end{align}
and for each $i < \omega$,
\begin{equation}
0 = r_i x + s_i y.\label{nconverse3}
\end{equation}
However, from \eqref{nconverse1}, we have $r_0 \neq 0$, thus $s_0 \neq 0$ by \eqref{nconverse3}. The same eqation then implies that $r_0$ is a multiple of $y$, therefore there cannot be a constant term on the right-hand side of \eqref{nconverse1}, a contradiction.
\end{example}

\medskip
We will continue by establishing some tools for proving that certain modules are \emph{not} very flat.

Our first tool is purely algebraic and employs the notion of an associated prime of a module \cite[\S 6]{M}:

Let $R$ be a noetherian ring, and let $Q$ denote its injective hull. Then
\begin{equation}\label{E(QbyR)}
E(Q/R) = \bigoplus_{p \in P} E(R/p)^{(\alpha_p)}
\end{equation} 
where $P = \Ass {R}{Q/R} \subseteq \Spec R$ and $\alpha_p = \mu_1(p,R)$ is the first Bass invariant of $R$ at $p$ (see \cite[\S9.2]{EJ}). 

For each $i \leq \mbox{Kdim}(R)$, we let $P_i$ denote the set of all prime ideals of height $i$. Since $P_0 \subseteq \Ass RR$, we have $P_1 \subseteq P$ by \cite[9.2.13]{EJ}. Of course, if $R$ is a noetherian domain of Krull dimension $1$, then $P = P_1$. In general, $\mbox{rad}(R) = \bigcap_{p \in P_0} p$ is the set of all nilpotent elements of $R$, while $Z(R) =  \bigcup_{p \in \Ass RR} p$ is the set of all zero-divisors of $R$. 

Let $s \in R$ be a non-zero divisor and $O(s) = \{ p \in P_1 \mid s \in p \}$. Then each $p \in O(s)$ is a minimal prime over $sR$, so the set $O(s)$ is finite. Moreover, for each $p \in P_1$, we have $R/p \otimes_R R[s^{-1}] = 0$, iff $p \in O(s)$. Indeed, $s \in p$ implies $(r + p) \otimes_R t/s^k =  (r.s + p) \otimes_R t/s^{k+1} = 0$, while if $s \notin p$, then $p \otimes_R R[s^{-1}]$ is a prime ideal in $R[s^{-1}]$, and $R/p \otimes_R R[s^{-1}] \cong R[s^{-1}]/(p \otimes_R R[s^{-1}]) \neq 0$.

\begin{lemma}\label{ass} Let $R$ be a noetherian domain. Let $M$ be a very flat module of finite rank $t$, and $F$ be its free submodule of the same rank. Then the set $P_1 \cap \Ass {R}{M/F}$ is finite.  
\end{lemma}
\begin{proof} By Lemma \ref{general}, there is $0 \neq s \in R$ such that $M \otimes_R R[s^{-1}]$ is a finitely generated $R[s^{-1}]$-module, whence $A = \Ass {R[s^{-1}]}{(M/F) \otimes_R R[s^{-1}]}$ is finite. 

Let $p \in P_1 \cap \Ass {R}{M/F}$, that is, $R/p \subseteq M/F$. If $p \notin O(s)$, then $R[s^{-1}]/(p \otimes_R R[s^{-1}]) \subseteq M/F \otimes_R R[s^{-1}]$, so $p \otimes_R R[s^{-1}] \in A$. It follows that $\card (P_1 \cap \Ass {R}{M/F}) \leq \card A  + \card O(s)$ is finite.
\end{proof}  

\medskip
Our next tool, the support of a module, comes from \cite{P}. We prefer the term \emph{P-support} here in order to distinguish it from the (different) standard notion of support used in commutative algebra, cf.\ \cite{EJ}, \cite{M}.
   
\begin{definition}\label{P-support}
For a module $M$ over a noetherian ring $R$, define its \emph{P-support} to be the set
\[ \PSupp M = \{p \in \Spec R \mid M \otimes_R k(p) \neq 0\}, \]
where $k(p)$ denotes the residue field of the prime ideal $p$.
\end{definition}

Note that for each ring homomorphism $f: R \to S$, the set $\PSupp (S)$ is the (underlying set of the) image of the induced scheme morphism $f^*: \Spec S \to \Spec R$.

The significance of P-support comes from the following:

\begin{lemma}\label{psupp}
The P-support of every very flat module is an open subset of $\Spec R$. Moreover, it is always nonempty, provided that the module is non-zero and $R$ is noetherian or reduced.
\end{lemma}
\begin{proof}
This follows directly from \cite[1.7.3--1.7.6]{P}.
\end{proof}

Lemma \ref{psupp} extends also to another kind of commutative coherent rings, the von Neumann regular ones. In fact, for those rings, all very flat modules are projective: 

\begin{example}\label{vonNeumann} Let $R$ be a von Neumann regular ring, that is, a ring such that for each $s \in R$ there is a (pseudo-inverse) element $u \in R$ such that $s = sus$, or equivalently, each module is flat. If $R$ is moreover commutative, then $R$ is unit regular, meaning that the pseudo-universe $u$ can always be chosen invertible in $R$, see \cite[4.2]{G}. 

For each $s \in R$, there is an $R$-isomorphism of $R[s^{-1}]$ on to $R[e^{-1}]$, where $e = su$, given by the assignment $r/s^i \mapsto ru^i/e^i$ (the inverse $R$-isomorphism maps $r/e^i$ to $r(u^{-1})^i/s^i)$. Moreover, $e = e^2$ is an idempotent, so we have the ring (and $R$-module) isomorphisms $R[e^{-1}] \cong R/(1-e)R \cong eR$. 

It follows that each very flat module is projective, isomorphic to a direct sum of cyclic projective modules generated by idempotents in $R$.
In particular, locally very flat modules coincide with (flat) Mittag-Leffler modules.

By \cite[3.2]{G}, $\Spec R = \mSpec R$. Let $e \in R$ be an idempotent and $p \in \Spec R$. Then $eR \otimes_R R/p = 0$, iff $e \in p$, whence $\PSupp {eR} = \{ p \in \Spec R \mid e \notin p \} = D(e)$. In general, if $M = \bigoplus_{i \in I} e_i R$, then 
$\PSupp {M}$ equals the open set $\bigcup_{i \in I} D(e_i)$. 

If $N$ is a submodule in a projective module $M$, then each non-zero finitely generated submodule of $N$ is a direct summand in $M$ (and hence in $N$), isomorphic to a finite direct sum of the form $\bigoplus_{i < n} e_iR$ for some non-zero idempotents $e_i \in R$, cf.\ \cite{G}. Let $E$ be the set of all idempotents $e \in R$ occuring in this way. Then $p \in \PSupp N$, iff $p \in \bigcup_{e \in E} \PSupp eR = \bigcup_{e \in E} D(e)$. It follows that the P-support of \emph{each} non-zero submodule of a very flat module forms a non-empty open subset of $\Spec R$.              
\end{example}

\medskip
For the rest of this section, we will restrict ourselves to the noetherian setting. In the following series of lemmas, the possibilities of constructing non-very flat modules via localizations of the ring are established.

\begin{lemma}\label{finitespectrum}
Let $R$ be a noetherian ring. Then the spectrum of $R$ is finite iff the set $P_1$ is finite (and the Krull dimension of $R$ is at most $1$).
\end{lemma}
\begin{proof}
Since the set $P_0$ is finite, the result follows directly from \cite[Theorem 144]{K}.
\end{proof}

\begin{lemma}\label{domainlike}
Let $R$ be a noetherian ring with infinite spectrum. Then there is $q_0 \in P_0$ such that the set
\[ Q_1 = \{p \in P_1 \mid \forall q \in P_0 : (q \subseteq p \Leftrightarrow q = q_0)\} \]
is infinite; moreover, there is an open set $W \subseteq \Spec R$ such that $W \cap P_0 = \{q_0\}$, $W \cap P_1 = Q_1$, and $W \subseteq \overline{\{q_0\}}$.
\end{lemma}
\begin{proof}
By Lemma \ref{finitespectrum}, the set $P_1$ is infinite. Suppose for the sake of contradiction, that none of the height-zero primes of $R$ fulfils the condition from the statement, i.e.\ for each $q \in P_0$, the set $T_q = P_1 \setminus \overline{P_0 \setminus \{q\}}$ is finite. This implies that $\{q\} = \Spec R \setminus \overline{T_q \cup (P_0 \setminus \{q\})}$ is an open set, and clearly a principal one. Therefore, there is $t_q \in R$ such that for each $p \in \Spec R$, $t_q \in p$ iff $p \neq q$. However, the ideal $I = \sum_{q \in P_0} t_q R$ is contained in each height-one prime, but in no height-zero prime, implying that there are inifinitely many minimal primes over $I$, a contradiction.

For the final claim, it suffices to put $W = \Spec R \setminus \overline{P_0 \setminus \{q\}}$.
\end{proof}

\begin{lemma}\label{infclosure}
Let $R$ be a noetherian ring with infinite spectrum and $q_0$, $Q_1$, $W$ as in Lemma \ref{domainlike}. Then the (Zariski) closure of any infinite subset of $Q_1$ contains $q_0$ (and consequently, the whole set $W$). In particular, the one-element set $\{q_0\}$ is not open.
\end{lemma}
\begin{proof}
Let $T \subseteq Q_1$ be infinite. Then the closure of $T$ are precisely those primes containing the ideal $I = \bigcap_{p \in T} p$. Since $R$ is noetherian, there are only finitely many minimal primes over $I$, so there have to be some height-zero ones among them. However, since $q_0 \subseteq I$, we see that $q_0$ is the only possible height-zero prime ideal over $I$. Therefore $I = q_0$ and the assertion follows.
\end{proof}

There is more to say for noetherian domains:

\begin{lemma}\label{gdomain}
Let $R$ be a noetherian domain. Then the following is equivalent:
\begin{enumerate}
\item[(i)]  The spectrum of $R$ is finite (and the Krull dimension of $R$ is at most $1$).
\item[(ii)]  $Q = R[s^{-1}]$ for some $s \in R$.
\item[(iii)]  Each flat module is very flat.
\end{enumerate}
\end{lemma}
\begin{proof}
(i) is equivalent to (ii): This is \cite[Theorem 146]{K} - note that the domains satisfying (ii) appear under the name \emph{G-domains} in \cite{K}.
(In fact, the implication (ii) implies (i) follows directly from Lemmas \ref{ass} and \ref{finitespectrum}, because $P_1 \subseteq P = \Ass R{Q/R}$.) 

(i) together with (ii) imply (iii): We have that $R$ is an almost perfect domain in the sense of \cite[7.55]{GT}, whence $^\perp (Q^\perp) = \mathcal F$ (see e.g.\ \cite[7.56]{GT}). The fact that $Q = R[s^{-1}]$ then implies $\VF = \mathcal F$.

(iii) implies (ii): This follows from Lemma \ref{general}(iii).
\end{proof}

Now we are ready to determine the conditions for the class $\VF$ to be covering.

\begin{lemma}\label{VFcoverlemma}
Let $R$ be a noetherian ring with infinite spectrum. Then the class $\VF$ is not covering.
\end{lemma}
\begin{proof}
Let $q_0$ be as in Lemma \ref{domainlike} and put $B = R_{q_0}$. Assume the existence of a $\VF$-cover $f\colon V \to B$. Pick $s \in R$ invertible in $B$, then $B \otimes_R R[s^{-1}] \cong B$, so we have a map $g_s\colon V \otimes_R R[s^{-1}] \to B$ together with the localization map $l_s\colon V \to V \otimes_R R[s^{-1}]$; clearly $f = g_s l_s$. Since $V \otimes_R R[s^{-1}]$ is a very flat module, by the (pre)covering property we have a map $b_s\colon V \otimes_R R[s^{-1}] \to V$ such that $g_s = f b_s$.

By the covering property, the map $b_s l_s$ is an automorphism of $V$, hence $l_s$ is a split inclusion. However, image of the localization map $l_s$ is essential in $V \otimes_R R[s^{-1}]$, thus in fact, $V \cong V \otimes_R R[s^{-1}]$. Since tensor product commutes with direct limits, by the above we have $V \cong V \otimes_R B$. However, by Lemma \ref{infclosure}, $\PSupp B = \{q_0\}$ is not an open set (nor it contains any non-empty open set), therefore $\PSupp V$ cannot be open and $V$ cannot be very flat in view of Lemma \ref{psupp}, a contradiction.
\end{proof}

\begin{theorem}\label{VFcoverdomain}
Let $R$ be a noetherian domain. Then the following is equivalent:
\begin{enumerate}
\item[(i)]  The class $\VF$ is covering.
\item[(ii)]  The spectrum of $R$ is finite (and the Krull dimension of $R$ is at most $1$).
\item[(iii)]  Each flat module is very flat.
\end{enumerate}
\end{theorem}
\begin{proof}
By Lemma \ref{VFcoverlemma}, we have that $\Spec R$ is finite, the rest is Lemma \ref{gdomain} together with the fact that the class of all flat modules is always covering (see e.g.\ \cite[8.1]{GT}). 
\end{proof}

\section{Locally very flat modules over noetherian rings}

Flat Mittag-Leffler modules coincide with the $\aleph_1$-projective modules (see e.g.\ \cite[\S3]{GT}). Replacing the term projective by very flat in the definition of an $\aleph_1$-projective module, we obtain the notion of a locally very flat module:  

\begin{definition}\label{lveryf} A module $M$ is said to be \emph{locally very flat} provided that it is locally $\mathcal C$-free where $\mathcal C$ enotes the class of all countably presented very flat modules (see Definition \ref{locally}).  

Note that a countably generated module is locally very flat, iff it is very flat. The class of all locally very flat modules is denoted by $\mathcal{LV}$. Clearly, $\mathcal{LV}$ consists of flat modules, and it contains all flat Mittag-Leffler modules.  
\end{definition}

\begin{example}
\def\Z{\mathbb Z}
The \emph{Baer-Specker groups} $\Z^\kappa$ ($\kappa \geq \omega$) are well-known not to be free, but they are flat Mittag-Leffler (\cite[3.35]{GT}), hence locally very flat. To see that they are not very flat, we use the refined version of Quillen's small object argument from \cite[Theorem 2]{ET} to obtain a short exact sequence
\[ 0 \to \Z \to C \to V \to 0 \]
with $V$ very flat and $C$ contraadjusted, both of cardinality at most $2^\omega$. As $C$ is an extension of very flat groups, it is very flat; as such, it cannot be cotorsion, for this would imply (by \cite[5.3.28]{EJ}) that the (non-zero torsion-free) $\Z_{(p)}$-module of all $p$-adic integeres $\mathbb{J}_{p}$ is a direct summand in $C$ for some prime $p$, in contradiction with Lemma \ref{ass}. Now \cite[1.2(4)]{GT1} implies that $\Ext1\Z{\Z^\omega}{C} \neq 0$. It follows that no Baer-Specker group is very flat.
\end{example}
 
We will distinguish two cases in our study of the approximation properties of the class $\mathcal{LV}$, depending on whether the set $\Spec R$ is finite or not:

\begin{lemma}\label{nprec}
Let $R$ be a noetherian ring such that $\Spec R$ is infinite. Then the class $\mathcal{LV}$ is not precovering.
\end{lemma}
\begin{proof} Since $\mathcal{LV}$ coincides with the class of all locally $\mathcal C$-free modules where $\mathcal C$ is the class of all countably presented very flat modules, $\mathcal{LV}$ fits the setting of Lemma \ref{saroch}. In view of that Lemma, it suffices to construct the appropriate Bass module~$B$.

Our goal is to construct $B$ as a direct limit of the direct system of the form
\[ R[s_0^{-1}] \to R[s_1^{-1}] \to \cdots \to R[s_k^{-1}] \to \cdots. \]
For $s \in R$, denote
\[ D_s = \{p \in \Spec R \mid s \notin p\} \]
the principal open set determined by $s$.
Let $q_0$, $Q_1$, $W$ as in Lemma \ref{domainlike}. We will construct the sequence $(s_k \mid k < \omega)$ such that $s_k \mid s_{k+1}$ and $s_k \notin q_0$ for $k < \omega$. First, let $s_0 \in R \setminus q_0$ be such that $D_{s_0}$ is a non-empty open subset of $W$. Assume that we have constructed $s_0, s_1, \dots, s_k$; since $s_k \notin q_0$, each $p \in Q_1$ such that $s_k \in p$ is a minimal prime over $s_k R$, therefore there are only finitely many such primes. Since $Q_1$ is infinite, we may pick $p_k \in Q_1$ such that $s_k \notin p_k$.
Finally, as $p_k \not\subseteq q_0$ and $s_k R \not\subseteq q_0$, we have $p_k \cap s_k R \not\subseteq q_0$ and we pick $s_{k+1} \in (p_k \cap s_k R) \setminus q_0$.

By construction, $p_k \in D_{s_k} \setminus D_{s_{k+1}}$, so by Lemma \ref{infclosure}, the interior of $\bigcap_{k<\omega} D_{s_k}$ is empty.
Since
\[ \PSupp B \subseteq \bigcap_{k<\omega} D_{s_k}, \]
we see that $B$ is not very flat by Lemma \ref{psupp} as desired.
\end{proof}

\begin{theorem}\label{char} Let $R$ be a noetherian domain. Then the following are equivalent:
\begin{enumerate}
\item The class $\LV$ is (pre)covering.
\item The spectrum of $R$ is finite and the Krull dimension of $R$ is at most $1$.
\item $\VF = \LV = \mathcal F_0$.
\end{enumerate}
\end{theorem}
\begin{proof}
The implication (i) $\Rightarrow$ (ii) is Lemma \ref{nprec}, the rest is just Theorem \ref{VFcoverdomain} together with the fact that each locally very flat module is flat.
\end{proof}

\section{Very flat and locally very flat modules over Dedekind domains} 

In this section, we will restrict ourselves to the case when $R$ is a Dedekind domain. Then $R$ is hereditary, so the class $\mathcal{VF}$ is closed under submodules, and $\mathcal{VF}$ coincides with the class of all $\mathcal S$-filtered modules, where $\mathcal S$ denotes the set of all non-zero submodules of the modules in $\mathcal L$. Moreover, if $\Spec R$ is finite, then $R$ is a PID, see \cite[p.86]{M}.        
  
\begin{proposition}\label{frank} Assume that $R$ is a Dedekind domain. Let $M$ be a torsion-free module of rank $t$. 
\begin{itemize}
\item[(i)] If $t = 1$, then $M$ is very flat, iff $M$ is isomorphic to a module in $\mathcal S$. 
\item[(ii)] Assume that $t$ is finite. Then $M$ is very flat, iff there exists $0 \neq s \in R$ such that $M \otimes_R R[s^{-1}]$ is a projective $R[s^{-1}]$-module of rank $t$. 
\item[(iii)] Assume that $t$ is finite and let $0 \to M^\prime \to M \to M^{\prime \prime} \to 0$ be a pure exact sequence of modules. Then $M$ is very flat, iff both $M^\prime$ and $M^{\prime \prime}$ are very flat.
\item[(iv)] $M$ is very flat, iff $M$ possesses an $\mathcal S$-filtration of length $t$. 
\end{itemize}   
\end{proposition}
\begin{proof} (i) If $M$ is very flat of rank $1$, then each $\mathcal S$-filtration of $M$ has length $1$, and the claim follows. 

(ii) The only if part is a particular instance of Lemma \ref{general}(ii). 

For the if part, note that $M \subseteq  M \otimes_R R[s^{-1}]$ as $R$-modules. By assumption, the latter is a projective $R[s^{-1}]$-module of finite rank, so it is finitely generated, hence a direct summand in $V = (R[s^{-1}])^n$ for some $n < \aleph_0$. Since $V$ is a very flat $R$-module, so is $M$. 

(iii) The if part holds because $\mathcal{VF}$ is closed under extensions. For the only-if part, we denote by $t$ the rank of $M$ and use (ii) to find $0 \neq s \in R$ such that $M \otimes_R R[s^{-1}]$ is a projective $R[s^{-1}]$-module of rank $t$. Localizing the original exact sequence at $R[s^{-1}]$, we obtain a pure-exact sequence of $R[s^{-1}]$-modules with a finitely generated projective middle term. The right hand term is a finitely generated flat, hence projective $R[s^{-1}]$-module, so the sequence splits, and (ii) yields the very flatness of both $M^\prime$ and $M^{\prime \prime}$.   

(iv) By the Eklof Lemma \cite[6.2]{GT}, each module possessing an $\mathcal S$-filtration is very flat. In order to prove the converse, let $\mathcal C$ denote the class of all countably presented very flat modules. We proceed in two steps: 

Step I. Assume that $M \in \mathcal C$, hence $t \leq \aleph_0$. We have $R^{(t)} \trianglelefteq M \trianglelefteq Q^{(t)}$. For each $n \leq t$, let $M_n = M \cap Q^{(n)}$. Then for each $n < t$, the module $S_n = M_{n + 1}/M_n$ is torsion-free of rank one, whence $M_n$ is a pure submodule of the finite rank very flat module $M_{n+1}$, for each $n < t$. By parts (i) and (iii), $S_n$ is isomorphic to an element of $\mathcal S$, so $M$ has an $\mathcal S$-filtration of length $t$. 

Step II: Let $M \in \mathcal{VF}$, $\lambda$ be the minimal cardinal such that $M$ is $\lambda$-presented, and assume that $\lambda > \aleph_0$. Let $\mathcal M$ be a $\mathcal C$-filtration of $M$ (see Lemma \ref{morep}). Let $\mathcal H$ be the family corresponding to $\mathcal M$ by Lemma \ref{hill} for $\kappa = \aleph_1$. Again, we have $R^{(t)} \trianglelefteq M \trianglelefteq Q^{(t)}$, and we let $\{ 1_\beta \mid \beta < t \}$ be the canonical free basis of $R^{(t)}$. Using the properties of the family $\mathcal H$, we can select from $\mathcal H$ by induction on $\beta$ a new $\mathcal C$-filtration $\mathcal M ^\prime = ( M^\prime_\beta \mid \beta \leq t \}$ such that $1_\beta \in M^\prime_{\beta + 1}$ for each $\beta < t$. Since $\mathcal H$ consists of pure submodules of $M$ and  $R^{(t)} \trianglelefteq M_t \trianglelefteq M$, we have $M_t = M$, so $\mathcal M ^\prime$ is a $\mathcal C$-filtration of $M$ of length $t$. Since $\mathcal C$ consists of countably presented modules, necessarily $t \geq \lambda$ (cf.\ \cite[Corollary 7.2.]{GT}). But clearly $t \leq \lambda$, so $t = \lambda$, and we can also assume that all the consecutive factors in $\mathcal M ^\prime$ are non-zero. Finally, by Step I, $0 \neq M_{\beta + 1}/M_\beta$ is countably $\mathcal S$-filtered for each $\beta < t$. Since the cardinal $t$ is uncountable, we can refine $\mathcal M ^\prime$ into an $\mathcal S$-filtration of $M$ of length $t$, q.e.d.                                            
\end{proof}

In the setting of Dedekind domains, the analogy between flat Mittag-Leffler modules and the locally very flat ones goes further: for example, Definition \ref{lveryf} can equivalently be formulated using pure submodules in $M$ (cf.\ \cite[3.14]{GT}), and one has the analog of Pontryagin's Criterion (in part (iii)): 

\begin{theorem}\label{variants} Let $R$ be a Dedekind domain and $M$ be a module. Then the following conditions are equivalent: 
\begin{itemize}
\item[(i)] For each finite subset $F$ of $M$, there exists a countably generated pure submodule $N$ of $M$ such that $N$ is very flat and contains $F$.
\item[(ii)] For each countable subset $C$ of $M$, there exists a countably generated pure submodule $N$ of $M$ such that $N$ is very flat and contains $C$.
\item[(iii)] Each finite rank submodule of $M$ is very flat.
\item[(iv)] Each countably generated submodule of $M$ is very flat.
\item[(v)] $M$ is locally very flat.
\end{itemize}
\end{theorem}
\begin{proof} (i) implies (ii): Let $C = \{ c_i \mid i < \omega \}$. By induction, we define a pure chain $\mathcal M  = ( M_i \mid i < \omega )$ of very flat submodules of $M$ of finite rank such that $\{ c_j \mid j < i \} \subseteq M_i$ for each $i < \omega$ as follows: $M_0 = 0$, and if $M_i$ is defined, then there is a finitely generated free submodule $G \trianglelefteq M_i +  c_iR$. By (i), there is also a countably generated pure submodule $D$ of $M$ such that $D$ is very flat and contains $G$. By Proposition \ref{frank}(iv), we can find a finite rank pure and very flat submodule $M_{i+1}$ of $D$ such that $G \subseteq M_{i+1}$, and hence also $M_i +  c_iR \subseteq M_{i+1}$. By Proposition \ref{frank}(iii), $M_{i+1}/M_i$ is very flat of finite rank, hence countably generated. Moreover, $\mathcal M$ is a $\mathcal{VF}$-filtration of $N = \bigcup_{i < \omega} M_i$. We conclude that $N$ is a countably generated very flat and pure submodule of $M$ containing the set $C$.       

(ii) implies (iii): Let $G$ be a finite rank submodule of $M$. Then $F \trianglelefteq G$ for a finitely generated free module $F$. By (ii), there is a countably generated very flat pure submodule $N$ of $M$ containing $F$. Then also $G \subseteq N$, whence $G$ is very flat.  

(iii) implies (iv): Let $C$ be a countably generated submodule of $M$ of countable rank. W.l.o.g., $R^{(\omega)} \trianglelefteq C \trianglelefteq Q^{(\omega)}$. 
For each $n < \omega$, let $C_n = C \cap Q^{(n)}$. By assumption, for each $n < \omega$, $C_n$ is a very flat pure submodule of $C$, whence $C_{n+1}/C_n$ is very flat by Proposition \ref{frank}(iii), and so is $C$. 

(iv) implies (v): If (iv) holds, then the set $\mathcal T$ of \emph{all} countably generated submodules of $M$ witnesses the local very flatness of $M$.
     
(v) implies (i): First, (v) clearly implies (iv), since each countably generated submodule of $M$ is contained in a (very flat) module from $\mathcal T$.  

In order to prove that (iv) implies (i), we let $F$ be a finite subset of $M$ and $G$ be a pure submodule of $M$ of finite rank, say $n$, such that $F \subseteq G$. Then  $R^{(n)} \trianglelefteq G \trianglelefteq Q^{(n)}$. It suffices to prove that $G$ is countably generated. 

If this is not the case, we let $G_i = G \cap Q^{(i)}$ for each $i \leq n$, and let $k < n$ be the largest index such that $G_k$ is countably generated (and hence very flat). Then $H = G_{k+1}/G_k$ is a torsion-free module of rank one, so w.l.o.g.\ $R \subseteq H \subseteq Q$, but $H$ is not countably generated. Hence $\Ass R{H/R}$ is uncountable. 

Let $\{ p_i \mid i < \omega \}$ be a set of distinct elements of $\Ass R{H/R}$. We can choose $g_0 \in G_{k+1}$ such that $g_0 + G_k = 1 \in R$, and for each $i < \omega$, $g_{i + 1} \in G_{k+1}$ such that $(\langle g_{i + 1} + G_k\rangle + \langle g_0 + G_k\rangle)/\langle g_0 + G_k\rangle = R/p_i \subseteq Q/R$. Let $G^\prime$ be the submodule of $G_{k+1}$ generated by $G_k \cup \{ g_i \mid i < \omega \}$. Since $G^\prime$ is countably generated, it is very flat, and so is its rank one pure-epimorphic image $H^\prime = G^\prime/G_k = \langle g_i + G_k \mid i < \omega \rangle$ (see Proposition \ref{frank}(iii)). By the definition of $H^\prime$, $R \subseteq H^\prime \subseteq H$, and $p_i \in \Ass R{H^\prime/R}$ for each $i < \omega$. So $\Ass R{H^\prime/R}$ is infinite, in contradiction with Lemma \ref{ass}.
\end{proof}

\section{Contraadjusted modules}\label{sect-CA}

\renewcommand{\Ext}{\operatorname{Ext}}
\renewcommand{\Hom}{\operatorname{Hom}}

Recall that a module $C$ is contraadjusted if $\Ext_R^1(R[s^{-1}], C) = 0$ for each $s \in R$.
This can be easily rephrased using the short exact sequence \eqref{present}:
\begin{lemma}\label{CAsystem}
A module $M$ is contraadjusted, if and only if for each $s \in R$ and for each sequence $( m_i \mid i < \omega )$ of elements of $M$, the countable system of linear equations with unknowns $x_i$
\begin{equation}\label{system}
x_i - sx_{i+1} = m_i \quad (i < \omega)
\end{equation}
has a solution in $M$.
\end{lemma}
\begin{proof}
Applying the contravariant functor $\Hom_R(-,M)$ to \eqref{present}, one sees that the condition $\Ext_R^1(R[s^{-1}], C) = 0$ is equivalent to the map $\Hom_R(g_s,M)$ being surjective. The latter condition easily translates into the solvability of the countable system \eqref{system}.
\end{proof}

\begin{example}
Let $R$ be a Dedekind domain. By \cite{N}, each reduced cotorsion module $C$ is isomorphic to the product $\prod_{p \in \mSpec R} C_p$, each $C_p$ being a module over the local ring $R_p$. Then $D = \bigoplus_{p \in \mSpec R} C_p$ is a contraadjusted module: To see it, pick $0 \neq s \in R$ and decompose $D$ as $D_1 \oplus D_2$, where
\[ D_1 = \bigoplus_{s \in p} C_p, \quad D_2 = \bigoplus_{s \notin p} C_p. \]
On one hand, since $D_1$ is a finite direct sum of cotorsion modules, it is cotorsion, so $\Ext_R^1(R[s^{-1}], D_1) = 0$. On the other hand, each summand in $D_2$ is $s$-divisible, so the system \eqref{system} has always a solution in $D_2$, and so $\Ext_R^1(R[s^{-1}], D_2) = 0$ and the assertion follows.

Finally, observe that $D$ is cotorsion only if $D \cong C$ (i.e.\ there are only finitely many non-zero summands $C_p$).
\end{example}

\begin{proposition}
If $R$ is a semiprime Goldie ring (e.g. a domain), then every divisible module (i.e.\ $sM = M$ for each non zero-divisor) is contraadjusted.
\end{proposition}
\begin{proof}
By \cite[9.1]{GT}, for semiprime Goldie rings, $\Ext_R^1(P, D) = 0$ whenever $P$ has projective dimension $\leq 1$ and $D$ is divisible, so the claim follows from Lemma \ref{morep}.
\end{proof}

If $M$ is a module and $0 \neq s \in R$, we let $\widehat{M}_s$ be the completion of $M$ in the ideal $sR$, i.e.\ the module
\[ \varprojlim\nolimits_{i<\omega} M/s^iM \]
(the maps between the modules being $m + s^{i+1} M \mapsto m + s^i M$). We further denote by $c_s$ the canonical morphism $M \to \widehat M_s$ sending $m$ to $(m + s^iM \mid i<\omega)$. The following lemma shows that the property of being contraadjusted can be translated to some form of completeness:

\begin{lemma}\label{CAcomplete}
Let $R$ be a ring, $M$ a module and $s \in R$. If $\Ext_R^1(R[s^{-1}], M) = 0$, then the canonical homomorphism $c_s$ is surjective. If $M$ has no $s$-torsion (i.e.\ $sm = 0 \Rightarrow m = 0$ for $m \in M$), then the reverse implication holds as well.
\end{lemma}
\begin{proof}
In the proof, we shall view $\widehat M_s$ as a submodule of the product $\prod_{i<\omega} M/s^iM$.

Assume the solvability of \eqref{system} and pick an element $(t_i + s^i M \mid i < \omega)$ in $\widehat M_s$. Put $m_0 = t_1$ and $m_i$ in $M$ such that $s^i m_i = (t_{i+1} - t_i)$ for $i>0$; such $m_i$'s exist because of the definition of inverse limit. Let $x_0, x_1, \dots$ be the solution of the system \eqref{system} with the given right-hand side $m_0, m_1, \dots$. It is now easy to check $x_0 - t_i \in s^iM$ for each $i < \omega$. Hence $x_0$ is the sought preimage of the element of the completion.

To show the converse, assume that $c_s$ is surjective and let $m_0, m_1, \dots$ be a sequence of elements of $M$; we shall check the solvability of the system \eqref{system}. In $\widehat M_s$, consider the element 
\[ \Bigl( \sum\nolimits_{k<i} m_k s^k + s^i M \mathrel{\Big|} i<\omega \Bigr); \]
let $x_0$ be any of its preimages in $c_s$. Now the elements $x_1, x_2, \dots$ can be simply constructed by a recurrence: By the definition of $x_0$, we have $x_0 - m_0 \in sM$, so there is $x_1 \in M$ such that $x_0 - s x_1 = m_0$. Given $x_1$, we observe that
\[ s(x_1 - m_1) = x_0 - m_0 - s m_1 \in s^2 M; \]
since $M$ has no $s$-torsion, we infer that $x_1 - m_1 \in s M$ and proceed as before to find $x_2, x_3, \dots$.
\end{proof}

The kernel of the homomorphism $c_s$ above is the intersection $\bigcap_{i<\omega} s^i M$, which is an $R[s^{-1}]$-module in case $M$ has no $s$-torsion. Thus, roughly said, there are two reasons for contraadjustedness of torsion-free modules: divisibility and completeness.

Our next goal will be to examine the existence of $\CA$-envelopes.

\begin{lemma}\label{CApreenevelope}
Let $M$ be an $R$-module, which is an $R[s^{-1}]$-module for some non-zero $s \in R$. Then there is a $\CA$-preenvelope of $M$ (in the category of $R$-modules), which is an $R[s^{-1}]$-module.
\end{lemma}
\begin{proof}
It suffices to construct a special $\CA$-preenvelope
\[ 0 \to M \to C \to V \to 0 \]
in the category of $R[s^{-1}]$-modules. By \cite[1.2.2]{P}, $C$ is a contraadjusted $R$-module. Likewise, $V$ is a very flat $R$-module because of \cite[1.2.3]{P}.
\end{proof}

\begin{lemma}\label{CAnotenv}
Let $R$ be a noetherian ring with infinite spectrum. Then the class $\CA$ is not enveloping.
\end{lemma}
\begin{proof}
Let $q_0$, $Q_1$ be as in Lemma \ref{domainlike} and pick $p_1, p_2 \in Q_1$ distinct. Put $N = S^{-1}R$, where $S = R \setminus (p_1 \cup p_2)$. Clearly, $N$ is a module over $R[s^{-1}]$ for each $s \in S$, so by Lemma \ref{CApreenevelope}, it has a $\CA$-preenvelope which is a module over $R[s^{-1}]$. If the $\CA$-envelope exists, it is a direct summand in each such preenvelope, hence an $N$-module.

Assume that $C$ is the $\CA$-envelope of $N$. By Wakamatsu lemma \cite[5.13]{GT}, $V = C/N$ is very flat; however, as a factor of $N$-modules, it is an $N$-module. Then, however, $V \cong V \otimes_R N$, so unless $V = 0$, we have $\PSupp V \subseteq \PSupp N$ and the latter set has empty interior in view of Lemma \ref{infclosure}, thus $V$ would not be very flat because of Lemma \ref{psupp}.

If $V = 0$, then $N \cong C$ is contraadjusted. Pick $t \in p_1 \setminus p_2$ and put $M = N/(S^{-1}q_0)$. Then $M$ as a factor of a contraadjusted module is contraadjusted. On the other hand, since $M$ as a ring is a noetherian domain, by the Krull intersection theorem, $\bigcap_{k<\omega} t^k M = 0$. Therefore if $\Ext_R^1(R[t^{-1}], M) = 0$, $M \cong \widehat M_t$. However, for each $r \in M$, the (image of the) element $t$ has an inverse in $\widehat M_{t}$ (namely $1 + rt + r^2t^2 + \cdots$), thus $t$ is in the Jacobson radical of $M$, and consequently in $p_2$, a contradiction.
\end{proof}

\begin{remark}
An analogous technique, i.e.\ constructing special precovers in the categories of $R[s^{-1}]$-modules, can be used to prove Lemma \ref{VFcoverlemma}.
\end{remark}

\begin{corollary}\label{caenveloping} Let $R$ be a noetherian domain. Then the condition that the class $\CA$ is enveloping is equivalent to all the other conditions of Theorem \ref{VFcoverdomain}.
\end{corollary}
\begin{proof}
This is just a direct application of Lemma \ref{CAnotenv} and the fact that the class of all cotorsion modules is always enveloping.
\end{proof}

{\bf Acknowledgement:} We would like to thank Leonid Positselski and Roger Wiegand for valuable comments and discussions. We are also very grateful to the referee for a number of comments and suggestions that helped to improve the paper.

\end{document}